\documentclass{article}
\usepackage[utf8]{inputenc}
\usepackage[english]{babel}
\usepackage{hyperref}
\usepackage{amsmath, amsfonts, amssymb}
\usepackage{enumitem}
\usepackage{stmaryrd}
\usepackage{graphicx}
\usepackage{color}
\usepackage{dsfont}
\usepackage[all]{xy}
\usepackage{csquotes}
\usepackage{subfigure}
\usepackage{geometry}
\usepackage{amsthm}
\usepackage{fullpage}

\usepackage{todonotes}
\usepackage{comment}

\theoremstyle{plain}

\usepackage[capitalize,nameinlink,noabbrev]{cleveref}
\newtheorem{theorem}{Theorem}

\newtheorem{proposition}[theorem]{Proposition}

\newtheorem{assumption}[theorem]{Assumption}
\newtheorem*{main-thm}{Main Theorem}
\crefname{assumption}{assumption}{assumptions}
\Crefname{assumption}{Assumption}{Assumptions}
\crefname{proposition}{proposition}{propositions}
\Crefname{proposition}{Proposition}{Propositions}
\theoremstyle{definition}

\usepackage[backend=bibtex,block=space,sortcites=true,maxbibnames=99,style=numeric]{biblatex}
\usepackage{doi}
\usepackage{appendix}
\usepackage{amsfonts}

\usepackage{nccmath}

\newtheorem{lem}{Lemma}

\newtheorem{rem}{Remark}

\newcommand{\bO}[1]{O\left(#1\right)}

\newcommand\Rp{\mathbb{R}_{+}}
\newcommand\R{\mathbb{R}}

\newcommand{\bigpar}[1]{\left({#1}\right)}
\def \l{\lambda}

\newcommand{\deriv}[2]{\frac{\partial {#1}}{\partial {#2}}}
\def\inv{^{-1}}
\def\gn{\frac{g}{n}}
\newcommand{\nto}{\ensuremath{n\rightarrow\infty}}
\def\ring{\ensuremath{K(\lambda)}}

\title{High genus one part monotone Hurwitz numbers}
\begin{document}

\author{Simon Barazer and Baptiste Louf}
\maketitle
\begin{abstract}
We obtain bivariate asymptotics for one part monotone Hurwitz numbers in high genus (i.e. as both the size and the genus go to infinity). To do so, we start with a linear recurrence for these numbers obtained by Do and Chaudhuri. Then, we apply a recent method developped by Elvey-Price, Fang, Wallner and the second author to extract asymptotics from such recurrences.
\end{abstract}
\section{Introduction}
\subsection{High genus: geometry and asymptotics}
Large genus geometry has been an active field of reseach for more than a decade now, as several communities investigated the asymptotic behavior of models of random surfaces as their genus tends to infinity. Such models include hyperbolic surfaces~\cite{MirzakhaniICM}, combinatorial maps~\cite{budzinski-louf} and flat surfaces~\cite{DGZZ}.

These geometric results are (often) obtained thanks to  asymptotic enumeration results (see for instance \cite{mirzakhani,aggarwal}). In a subset of these works, besides the genus $g$, a size parameter $n$ also goes to infinity, hence the enumeration problem considered belongs to the field of multivariate asymptotics, where the current knowledge is substantially more limited than in the univariate case (we refer the reader to~\cite{acsv_web} for a systematic approach).

In the case of enumerative geometry models, generating series are often solutions of integrable hierarchies~\cite{okounkov,KP-triangulation}, which sometime yields recurrence formulas for the models considered. In specific cases, these recurrence formulas become linear. For such formulas, a recent method was introduced by  Elvey-Price, Fang and Wallner, together with the second author~\cite{EFLW}. In this work we will apply this method to monotone Hurwitz numbers.

\subsection{Monotone Hurwitz numbers}

Monotone Hurwitz numbers count certain classes of branched covers of the sphere by higher genus surfaces, or alternatively, factorisations in the symmetric group that follow a certain monotonicity rule. They were introduced in \cite{GGPN} to provide a combinatorial expansion of the HCIZ integral.

In this work we will focus on one part monotone Hurwitz numbers. The number $m_g(d)$ counts monotone factorisations of a long cycle in $\mathfrak S_d$ into $d+2g-1$ transpositions, divided by $d!$. The monotonicity condition imposes that if the transposition $(a,b)$ appears after the transposition $(c,d)$ in the product, then $\max(a,b)\geq \max(c,d)$. 

In \cite{DoChaudhuri}, a linear bivariate recurrence for the numbers $m_g(d)$ has been obtained, among other similar recurrences, in the fashion of the formula obtained by Harer and Zagier to count combinatorial maps with one face~\cite{HarerZagier1986Euler}.
\begin{equation}
    dm_g(d)=2(2d-3)m_g(d-1)+d(d-1)^2m_{g-1}(d)
\end{equation}
For convenience, we will write $E(n,g):=m_g(n+1)$, the equation above becomes
\begin{equation}\label{eq_recursion}
 (n+1)E(n,g)=2(2n-1)E(n-1,g)+n^2(n+1)E(n,g-1).
\end{equation}
\subsection{Result}
Our main result is to obtain asymptotics for the monotone Hurwitz numbers $E(n,g)$ as both $n$ and $g$ go to infinity.
\begin{theorem}\label{thm_main}
As $\nto$, for any sequence $g=g_n$, we have
\begin{equation}
    E(n,g)\sim \frac{\sqrt{g}g^g}{\sqrt{2\pi}e^gg!}n^{2g-2}\exp\left(nf\bigpar{\frac{g}{n}}+j\bigpar{\gn}\right)
\end{equation}
where $f$ and $j$ are explicit functions defined in section~\ref{sec_def}.
\end{theorem}

\begin{rem}
    If $g\to \infty$, the formula above simplifies to
    \[E(n,g)\sim \frac{1}{2\pi}n^{2g-2}\exp\left(nf\bigpar{\frac{g}{n}}+j\bigpar{\gn}\right).\]
\end{rem}

In order to prove this result, we will use a new method to obtain bivariate asymptotics from the analysis of linear recurrences that was developped by the second author together with Elvey-Price, Fang and Wallner in~\cite{EFLW}. Roughly speaking, this method consists in a ``guess-and-check approach". The checking part relies on modeling the recurrence by a well chosen random walk.

\begin{rem}
    The paper~\cite{DoChaudhuri} contains other recurrences for other models of enumerative geometry (restricted to one part). However, the coefficients of these recurrences are not all positive, which is required by the random walk method of~\cite{EFLW}. Nevertheless, it is expected that, in these other models, the bivariate asymptotics are of the same flavor as that of Theorem~\ref{thm_main}.
\end{rem}

\section{Definitions and heuristic guessing}\label{sec_def}

\subsection{Heuristic guessing}
We consider the approximation form in the statement of Theorem \ref{thm_main}, namely:
\begin{equation}
\label{formula_omega}
    \Omega(n,g)=\frac{\sqrt{g}g^g}{\sqrt{2\pi}e^g g!}n^{2g-2}\exp(nf(\frac{g}{n})+j(\frac{g}{n})).
\end{equation}
We will denote $C_g=\frac{\sqrt{g}g^g}{\sqrt{2\pi}e^g g!}$. Our plan is to insert formula \ref{formula_omega} in equation \ref{eq_recursion}, and divide both sides by $\Omega(n,g)$. Assuming that $\frac{g_n}{n}\to \theta$, we expand the expression using Taylor's formula up to the order $o(1)$, and we obtain an equation for $f$. We have the formula
\begin{equation*}
   \lim_{n\to \infty} \frac{\Omega(n-1,g_n)}{\Omega(n,g_n)}=\lambda(\theta)~~~\text{and}~~~\lim_{n\to \infty}n^2\frac{\Omega(n,g_n-1)}{\Omega(n,g_n)}=\exp(-f'(\theta)),
\end{equation*}
with 
\begin{equation}
\label{formula_lambda}
\lambda=\exp(-2\theta-f+\theta f').
\end{equation}
By using these formulas, in order to have the condition
\begin{equation*}
    \frac{2(2n-1)}{n+1}\frac{\Omega(n-1,g_n)}{\Omega(n,g_n)}+n^2 \frac{\Omega(n,g_n-1)}{\Omega(n,g_n)}=1+o(1),
\end{equation*}
we see that $f$ should satisfy the following differential equation:
\begin{equation}\label{eq_heuristic_f}
    1=4\lambda+\exp(-f').
\end{equation}
Applying logarithmic derivatives to this equation and formula~\ref{formula_lambda}, we obtain the two equations:
\begin{equation}\label{eq_f''}
    f''=\frac{4\lambda'}{1-4\lambda}~~~\text{and}~~~\theta f''=2+\frac{\lambda'}{\lambda}.
\end{equation}
Now, we can eliminate $f''$ to obtain a single equation in $\lambda$:
\begin{equation}
\label{eq_lambda}
    2+\frac{\lambda'}{\lambda}=\frac{4\theta \lambda'}{1-4\lambda}\Longleftrightarrow \lambda'=-\frac{2\lambda(1-4\lambda)}{1-4(\theta+1)\lambda}.
\end{equation}
We can solve \ref{eq_recursion} for $g=0$, in this case 
\begin{equation*}
    \frac{E(n-1,0)}{E(n,0)}\to \frac{1}{4},
\end{equation*}
then, using equation \ref{formula_lambda}, we obtain the following initial condition for $\lambda$:
\begin{equation*}
    \lambda(0)=\frac{1}{4}.
\end{equation*}
A heuristic analysis of equation \ref{eq_lambda} suggests that $\lambda$ is strictly decreasing. Let $\theta: ]0,\frac{1}{4}[\to \R_+$ be the inverse of $\lambda$, assuming it exists; we have $\theta'(\lambda)=\frac{1}{\lambda'(\theta(\lambda))}$. We invert equation \ref{eq_lambda} and obtain the following linear equation for $\theta$, as a function of $\lambda$:
\begin{equation}
\label{eq_theta}
    \frac{d\theta}{d\lambda}=\frac{2\theta}{1-4\lambda}-\frac{1}{2\lambda}.
\end{equation}
Similarly, we can do the same heuristic for $j$, this time we have to take higher terms in the Taylor expansion, we obtain the following:
\begin{equation}\label{eq_heuristic_j}
    (4\theta\lambda -\exp(-f'))j'+\frac{f''}{2}(4\theta^2\lambda+\exp(-f'))+4\lambda\left (\frac{1}{2}-\theta\right )=0.
\end{equation}
Which can be rewritten as 
\begin{equation*}
    4\lambda \left (\frac{1}{2}-\theta+\frac{\theta^2f''}{2}+\theta j'\right)+\exp(-f')\left( \frac{f''}{2}-j'\right)=0.
\end{equation*}
\begin{rem}[Coefficient $n^{2g-2}$]
    The coefficient $2$ in front of $g$ in $n^{2g-2}$ is due to the scaling behavior of the coefficient in the recursion; it allows to kill the weight $n^2$ in the second term. The coefficient $-2$ is a correction to match with low values of $g$ and can't be determined by equation \ref{eq_heuristic_j}.
\end{rem}
\subsection{Rigorous definition}

\begin{proposition}\label{prop_existence_theta}
    Equation \eqref{eq_theta} has a unique solution satisfying the initial condition $\theta(1/4)=0$, with explicit formula
  \begin{equation*}\label{formula_theta}
     \theta(\lambda)= -1+\frac{\text{Artanh}(\sqrt{1-4\lambda})}{\sqrt{1-4\lambda}}
  \end{equation*}
    What's more $\theta$ is strictly decreasing and tend to $\infty$, when $\lambda$ goes to $0$.
\end{proposition}

\begin{proof}
    The equation \ref{eq_theta} is linear, we have the initial condition $\lambda(0)=\frac{1}{4}$ and then $\theta\left(\frac{1}{4}\right)=0$. The unique solution, is given at the neighborhood of $\frac{1}{4}$ by 
    \begin{equation*}
       \theta(\lambda)=\frac{1}{\sqrt{1-4\lambda}} \int_{\lambda}^{\frac{1}{4}}\frac{\sqrt{1-4x}}{x}dx=-1+\frac{\text{Artanh}(\sqrt{1-4\lambda})}{\sqrt{1-4\lambda}}.
    \end{equation*}
    Then, it is straightforward to see that this solution extends to $\left]0,\frac{1}{4}\right[$. Moreover, using the properties of $x\to \frac{\text{Artanh}(x)}{x}$, the function $\theta$ is analytic and strictly decreasing on $\left]0,\frac{1}{4}\right[$.
\end{proof}
Thanks to what's above, for all $\theta\in[0,\infty)$, we can associate $\lambda(\theta)$ as the inverse of formula \ref{formula_theta} and it satisfies equation \ref{eq_lambda}.
Accordingly with the previous subsection, we define
\begin{equation}\label{eq_def_f}
    f=-\ln(\lambda)-2\theta-\theta \ln(1-4\lambda).
\end{equation}

\begin{proposition}\label{prop_existence_f}
    The function $f$ satisfies~\eqref{eq_heuristic_f}.
\end{proposition}
\begin{proof}
    The only thing to check is that $f$ is well defined, which is true because $\lambda\in \left]0,\frac{1}{4}\right[$. The fact that $f$ solves equation \ref{eq_heuristic_f} is straightforward by construction. 
\end{proof}
We define:
\begin{equation}\label{eq_def_j}
    j=-\frac{\ln(1-4(\theta+1)\lambda)}{2}+\frac{\ln(2)}{2}.
\end{equation}

\begin{proposition}\label{prop_existence_j}
    The function $j$ is well-defined and satisfies~\eqref{eq_heuristic_j}.
\end{proposition}

\begin{proof}
First of all we show that the function is well defined. The function $\lambda$ is well defined for $\theta\in ]0,+\infty[$ and differentiable; moreover, we have the differential equation:
\begin{equation*}
    \lambda'=-\frac{2\lambda(1-4\lambda)}{(1-4(\theta+1)\lambda)}.
\end{equation*}
Then, if $g(\theta)=1-4(\theta+1)\lambda$, we must have $g(\theta)>0$ for $\theta>0$. According to equation \ref{eq_heuristic_f}, $1-4\lambda=\exp(-f')$, and then 
\begin{equation*}
    4\lambda \theta -\exp(-f')=4(\theta+1)\lambda-1=-g(\theta).
\end{equation*}
Taking the derivative of the LHS with respect to $\theta$, and using $\lambda'=\lambda(\theta f''-2)$, we obtain
\begin{eqnarray*}
   -g'(\theta)= 4\lambda+4\lambda'\theta+f''\exp(-f')&=&4\lambda+4\lambda \theta^2 f''-8\lambda\theta+f''\exp(-f')\\
   &=&f''(4\lambda\theta^2+\exp(-f'))+4\lambda(1-2\theta).
\end{eqnarray*}
    Rewriting equation \ref{eq_def_j}, we get 
    \begin{equation*}
        j'(\theta)g(\theta)+\frac{g'(\theta)}{2}=0.
    \end{equation*}
    We then obtain proposition \ref{prop_existence_j}.
\end{proof}


\subsection{Asymptotic properties in $\theta=0$}
\begin{proposition}\label{prop_asym_zero}
    We have
    \begin{equation*}
        f(\theta)=-\theta\ln(\theta)+f_0(\theta)~~~\text{and}~~~j(\theta)=-\frac{\ln(\theta)}{2}+j_0(\theta),
    \end{equation*}
    where $f_0$ and $j_0$ are analytic on the neighborhood of $0$ and on $]0,+\infty[$. Moreover, we have the formulas:
        \begin{equation*}
        f_0(0)=\ln(4),~~~~~\exp(f_0'(0))=\frac{e}{3},~~\text{and}~~j_0(0)=0.
    \end{equation*}
    Finally, in the neighborhood of $0$, and for $k\ge 1$, we have:
    \begin{equation*}
        f^{(k+1)}(\theta)=\bO{\theta^{-k}},~~\text{and}~~j^{(k)}(\theta)=\bO{\theta^{-k}}
    \end{equation*}
\end{proposition}

\begin{proof}
    According to formula \ref{formula_theta} it seems that $\theta$ as a singularity at zero, nevertheless, we have 
    \begin{equation*}
    \frac{\text{Arctanh}(x)}{x}-1=\sum_{k=1}^{\infty} \frac{x^{2k}}{2k+1}.
    \end{equation*}
    Then, in formula \ref{formula_theta}, the square root disappears and we obtain that $\theta$ is analytic at the neighborhood of $\lambda=\frac{1}{4}$, with a non-vanishing first derivative. Indeed 
    \begin{equation*}
        \theta(\lambda)=-\frac{4}{3}(\lambda-\frac{1}{4})+o\left (\lambda-\frac{1}{4}\right)
    \end{equation*}
    Then, using the local inversion theorem for analytic functions, $\lambda$ is also analytic at $\theta=0$, moreover, $\lambda(0)=\frac{1}{4}$ and $\lambda'(0)=-\frac{3}{4}$. We obtain 
    \begin{equation*}
        \lambda(\theta)=\frac{1}{4}-\frac{3\theta}{4}+o(\theta).
    \end{equation*}
     We now use formula \ref{eq_def_f} for $f$, all the terms are analytic at $0$ except
    \begin{equation*}
        \theta\ln(1-4\lambda)=\theta\ln(3\theta+o(\theta))=\theta\ln(\theta)+``analytic".
    \end{equation*}
    Because the error term is an analytic function, then, we obtain the desired form. To obtain the precise values, using equation \ref{eq_heuristic_f}, we have $\theta\exp(1-f_0'(\theta))=1-4\lambda$ then $\exp(1-f_0'(0))=-4\lambda'(0)$, and then 
    \begin{equation*}
        \exp(f_0'(0))=\frac{e}{3}.
    \end{equation*}
    Using formula \ref{eq_def_f} we get
    \begin{equation*}
        f=-2\theta-\ln(1-4\lambda)\theta-\ln(\lambda).
    \end{equation*}
    We also obtain $f_0(0)=\ln(\lambda(0))=\ln(4)$. 
    We treat $j$ in a similar way; we have
    \begin{equation*}
        \ln(1-4(\theta+1)\lambda)=\ln(1-(\theta+1)(1-3\theta)+o(\theta))=\ln(\theta(2+o(1)))=\ln(\theta)+\ln(2)+o(1),
    \end{equation*}
     the error term is analytic, and we obtain proposition \ref{prop_asym_zero} by dividing by $-2$ and adding the constant.
\end{proof}

\subsection{Asymptotic properties in $\theta=\infty$}

The main purpose of proposition~\eqref{prop_asym_zero} was to obtain asymptotics for $f$, $j$ and their derivatives in $\theta=0$ without having to do to many calculations: the advantage of dealing with analytic functions was to be able to ``differentiate small $o$'s" (something that is in general not allowed).

At $\theta=\infty$, things are slightly more complicated: first, we will have to expand in $\lambda=0$ first, not directly in $\theta$. And also, we will use a slighlty more complicated ring that still allows differentiation of small $o$'s. We describe it now.

Let $\ring$ be the smallest field containing the function $\log \lambda$ as well as all the power series in $\lambda$ with real coefficients and non zero radius of convergence. 

\begin{proposition}\label{prop_nice_ring}
The following properties hold:
\begin{itemize}
    \item for every $g\in\ring$, there exists $\epsilon>0$ such that $g$ can be seen as a $C^\infty$ function for $\lambda\in(0,\varepsilon)$;
    \item $\ring$ is a differential field, i.e., it is stable under differentiation;
    \item if $g,h\in\ring$ are such that, as $\lambda\to 0$, $g(\lambda)=o(h(\lambda))$ and $h(\lambda)\neq  \Theta(1)$, then $g'(\lambda)=o(h'(\lambda))$.
\end{itemize}
\end{proposition}

Note that the first two points are needed for the third point to even make sense. This proposition follows rather directly from the theory of Hardy fields (up to considering the change of variables $x=1/\lambda$ to work at $+\infty$), see for instance \cite{hardy-fields}.

With this property in hand, asymptotics of $f$ and $j$ at $+\infty$ follow easily.
\begin{proposition}\label{prop_asym_infty}
    For every $0<a<2$, as $\theta\to \infty$ we have
   \begin{equation*}
            f'''(\theta)=\bO{\exp(-a\theta)}~~~\text{and}~~~j''(\theta)=\bO{\exp(-a\theta)}.
        \end{equation*}
\end{proposition}

\begin{proof}
    Close to $\lambda=0$, one can write 
    \begin{equation*}
    \theta(\lambda)=-\frac{\ln(\lambda)}{2\sqrt{1-4\lambda}}+``analytic",
\end{equation*}

Therefore $\theta(\lambda)\in\ring$, and it follows easily by equations~\eqref{eq_def_f} and~\eqref{eq_def_j} that $f(\theta(\lambda)),j(\theta(\lambda))\in\ring$.
Hence, one can successively deduce the following expansions in $\l=0$:
\[\theta(\lambda)=\frac{-\log(\l)}{2}+cst+O(\l\log\l);\]
\[\deriv{\theta}{\l}=\frac{-1}{2\l}+o(1);\]
\[f(\theta(\l))=cst+O(\l\log\l)\]
\[f^{(k)}(\theta)=\bigpar{\deriv{\theta}{\l}}\inv\times \deriv{f^{(k-1)}(\theta)}{\l} = O(\l)\times O(\log \l)=O(\l\log\l) ;\]
\[j(\theta(\l))=cst+O(\l\log\l)\]
\[j^{(k)}(\theta)=\bigpar{\deriv{\theta}{\l}}\inv\times \deriv{j^{(k-1)}(\theta)}{\l} = O(\l)\times O(\log \l)=O(\l\log\l) ;\]

Now, as $\theta\to\infty$, $\l\to 0$ and by the first expansion above we have
\[\l(\theta)\log\l(\theta)=O(\exp(-a\theta)),\]
which entails the result.
\end{proof}

\section{Proof ideas}
In \cite{EFLW}[Theorem 19], a general result is given to check that the guessed asymptotics are indeed correct. Underlying its proof is an associated random walk, but it will be hidden in this article, we will only verify the needed assumptions.

We set 
\begin{equation}\label{eq_def_Omega}
    \Omega(n,g)=\frac{g^g\sqrt{ g}}{\sqrt{2\pi}g!e^g}n^{2g-2}\exp\left(nf\left(\frac{g}{n}\right)+j\left (\frac{g}{n}\right )\right).
\end{equation}
for $g\geq 1$, and 
\begin{equation}\label{eq_def_Omega_zero}
  \Omega(n,0)= \frac{4^nn^{-\frac{3}{2}}}{\sqrt{2\pi}}.
\end{equation}
Then we can define auxiliary functions
\begin{equation}
     s(n,g)=\frac{\Omega(n,g-1)}{\Omega(n,g)};
\end{equation}

\begin{equation}
    \alpha(n,g)=\frac{2(2n-1)}{n+1}\frac{\Omega(n-1,g)}{\Omega(n,g)}\quad \text{and}\quad \beta(n,g)=n^2\frac{\Omega(n,g-1)}{\Omega(n,g)};
\end{equation}

\begin{equation}
    Q(n,g)=\frac{E(n,g)}{\Omega(n,g)}.
\end{equation}

Let us now state our main assumptions. The first one amounts to saying that the numbers $\Omega(n,g)$ satisfy the recurrence~\eqref{eq_recursion} ``asymptotically" (with a sufficient precision).

\begin{assumption}\label{assum_alpha_beta}
There exists a summable function $\eta$ such that, as $\nto$,  uniformly in $g$.
\[\alpha(n,g)+\beta(n,g)=1+O(\eta(n+g)).\]
\end{assumption}

Then, we want to control ``boundary values" of $\Omega(n,g)$ (in~\cite{EFLW}, the following condition was named ``asymptotic initial condition").
\begin{assumption}\label{assum_boundary_conditions}
    As $\nto$, 
\[Q(n,0)\to 1,\]
and there exists a constant $C>0$ such that, for all $g$,
\[Q(1,g)<C.\]
\end{assumption}

We also need to make sure that the underlying random walk behaves ``well", which is encoded in the behavior of the function $s$. First we require some boundary conditions on $s$.
\begin{assumption}\label{assum_s}
    For all $n\geq 1$, 
    \[s(n,1)>0,\]
    and there exists $c>0$ such that for all $g\geq 1$, 
    \[s(2,g)>c.\]
\end{assumption}
And finally we wish to control its asymptotic behavior.

\begin{assumption}\label{assum_good_starting_point}
    For any sequence $g=g_n$, as $\nto$
    \[s(n,g_n)\to 0.\]
\end{assumption}

\begin{rem}
    In the language of~\cite{EFLW}, we have set here:
    $\mathcal{B}^{good}=\{(n,0)|n\geq 1\}$, $\mathcal{B}^{bad}=\{(1,g)|g\geq 1\}$ and $\mathcal{I}=\{(n,g)|g\geq 1,n\geq 2\}$.
\end{rem}

Provided that the assumptions above are satisfied, by ~\cite{EFLW}[Theorem 19], we have proven what we wanted:
\begin{theorem}\label{thm_general_RWs}
    If the assumptions above are satisfied, then $E(n,g)\sim\Omega(n,g)$ and theorem~\ref{thm_main} holds.
\end{theorem}
In the next section, we will prove that these assumptions are indeed satisfied.

\section{Proof of the assumptions}
We introduce the following change of variables:
\begin{equation*}
    \theta=\frac{g}{n}~~~\text{and}~~~x=n+g.
\end{equation*}
the old variables $(n,g)$ can be recovered by
\begin{equation*}
    n=\frac{x}{1+\theta}~~~\text{and}~~~g=\frac{x\theta}{1+\theta}.
\end{equation*}
\subsection{$\alpha+\beta=1$}
\begin{proposition}\label{prop_alpha_beta}
Uniformly on $\theta$, assuming $g\ge 1$ and $n\ge 2$, we have 
    \begin{equation*}
        \alpha+\beta=1+\bO{\frac{1}{x^{\frac{4}{3}}}}
    \end{equation*}
\end{proposition}
In order to prove the proposition \ref{prop_alpha_beta} we distinguish two cases:
\begin{itemize}
    \item Low genus, when $\theta\le x^{-\frac{2}{3}}$, where we use proposition \ref{prop_asym_zero}.
     \item High genus, when $\theta \ge x^{-\frac{2}{3}}$,and  in this case we also use proposition \ref{prop_asym_infty}.
\end{itemize}

\subsubsection{Low genus}
Let us first rewrite $\Omega(n,g)$ when $g=o(n)$. By using proposition \ref{prop_asym_zero}, we have the following modified guess:
\begin{eqnarray*}
    \Omega(n,g)&=&C_g n^{2g-2}\exp\left(-g\ln(g)+g\ln(n)+\frac{\ln(g)}{2}-\frac{\ln(n)}{2}+nf_0(\theta)+j_0(\theta)\right)\\
    &=&\frac{C_g}{g^g \sqrt{g}}n^{3g-\frac{3}{2}}\exp\left (nf_0(\theta)+j_0(\theta)\right )\\
    &=&\tilde{C}_g n^{3g-\frac{3}{2}}\exp\left (nf_0(\theta)+j_0(\theta)\right )   ,
\end{eqnarray*}
where we introduce $\tilde{C}_g=\frac{C_g}{g^g \sqrt{g}}=\frac{1}{\sqrt{2\pi} g! e^g}$.
\begin{rem}[Fixed $g$.]
    In particular, for fixed $g$ (including $g=0$ !), we have 
\begin{equation*}
     \Omega(n,g)\sim \frac{4^nn^{3g-\frac{3}{2}}}{g! e^g\sqrt{2\pi}}.
\end{equation*}
\end{rem}
We start with the following lemma:
\begin{lem}
\label{lem_taylor_low}
    Assume that $\theta=o(1)$ and $\delta=\bO{\theta}$, we have the approximations:
    \begin{eqnarray*}
        f_0(\theta+\delta)-f_0(\theta)&=&f'_0(0)\delta +\bO{\theta \delta},\\
        j_0(\theta+\delta)-j_0(\theta)&=&\bO{\delta}.
    \end{eqnarray*}
    Moreover, we have the formulas:
    \begin{equation*}
        f_0(0)=\ln(4),~~~~~\exp(f_0'(0))=\frac{e}{3},~~\text{and}~~j_0(0)=0.
    \end{equation*}
\end{lem}
\begin{proof}
    The first part is a direct consequence of the analyticity of $j_0$ and $f_0$ and the use of Taylor formula. For instance
    \begin{equation*}
        f_0(\theta+\delta)-f_0(\theta)=f_0'(\theta)\delta+\bO{\delta^2}=f_0'(0)\delta+\bO{\theta\delta}.
    \end{equation*}
    
\end{proof}

 Using this, we prove the following:
\begin{lem}
\label{lem_low}
    If $\theta \le x^{-\frac{2}{3}}$, we have 
    \begin{eqnarray*}
    \alpha(n,g)&=&1-3\theta+\bO{\frac{1}{x^{\frac{5}{3}}}},\\
    \beta(n,g)&=&3\theta+\bO{\frac{1}{x^{\frac{4}{3}}}},
    \end{eqnarray*}
    the error term is uniform in $\theta$.
\end{lem}

\begin{proof}
    Let $\delta=\frac{g}{n-1}-\frac{g}{n}=\frac{g}{n(n-1)}=\frac{\theta}{n-1}=\bO{\frac{1}{x^{\frac{5}{3}}}}$, according to lemma \ref{lem_taylor_low}, we get
    \begin{equation}
        \label{formula_alpha_low}
        \frac{2(2n-1)}{n+1}\frac{\Omega(n-1,g)}{\Omega(n,g)}=\left(4-\frac{6}{n}+\bO{\frac{1}{n^2}}\right)\left (1-\frac{1}{n}\right)^{3g-\frac{3}{2}}\exp\left ((n-1)f_0(\theta+\delta)-nf_0(\theta))+j_0(\theta+\delta)-j_0(\theta)\right ),
    \end{equation}
    the second term is given by 
    \begin{equation*}
        \left (1-\frac{1}{n}\right)^{3g-\frac{3}{2}}=\exp\left(\left(3g-\frac{3}{2}\right)\ln\left(1-\frac{1}{n}\right)\right)=\exp \left( -3\theta +\frac{3}{2n}+\bO{\frac{\theta}{n}}\right)=1-3\theta +\frac{3}{2n}+\bO{\frac{\theta}{n}
        }.
    \end{equation*}
    Using lemma \ref{lem_taylor_low}, the last term of formula \ref{formula_alpha_low}, is equal to
    \begin{equation*}
        \exp \left(f_0'(0)\theta-f_0(0)-f'_0(0)\theta+\bO{\delta}\right)=\exp(-f_0(0)) \left(1+\bO{\frac{\theta}{n}}\right)=\frac{1}{4}\left(1+\bO{\frac{\theta}{n}}\right).
    \end{equation*}
    Putting this together and the fact that, in this range $\frac{1}{n^2}=\bO{\frac{\theta}{n}}$, we finally obtain the first part of lemma \ref{lem_low}:
    \begin{eqnarray*}
        \alpha(n,g)&=&\left(1-\frac{3}{2n}+\bO{\frac{1}{n^2}}\right)\left( 1-3\theta+\frac{3}{2n}+\bO{\frac{\theta}{n}}\right )\left(1+\bO{\frac{\theta}{n}}\right)\\
        &=&1-3\theta+\bO{\frac{\theta}{n}}\\
     &=&1-3\theta+\bO{\frac{1}{x^{\frac{5}{3}}}}.
    \end{eqnarray*}
     We proceed similarly for $\beta$, let $\delta=-\frac{1}{n}=\bO{\theta}$, then according to lemma \ref{lem_taylor_low}:
    \begin{equation*}
        \beta(n,g)=\frac{\tilde{C}_{g-1}}{\tilde{C}_{g}}\frac{1}{n}\exp\left(n(f_0(\theta+\delta)-f_0(\theta))+j_0(\theta+\delta)-j_0(\theta)\right).
    \end{equation*}
  We obtain $ \frac{\tilde{C}_{g-1}}{\tilde{C}_{g}}=ge$. Using lemma \ref{lem_taylor_low}, we can write 
    \begin{equation*}
        \beta(n,g)=e\theta\exp(-f'_0(0)+\bO{\theta}).
    \end{equation*}
    Finally, according to lemma \ref{lem_taylor_low}, we know $\exp(f_0'(0))=\frac{e}{3}$ and we finally obtain the second part of lemma \ref{lem_low}
    \begin{equation*}
    \beta(n,g)=3\theta+\bO{\theta^2}=3\theta+\bO{\frac{1}{x^{\frac{4}{3}}}}.
    \end{equation*}

\end{proof}

\subsubsection{Intermediate and high genus}
In this part we assume $\theta \ge x^{-\frac{2}{3}}$. We will use the following lemma. 
\begin{lem}
\label{lem_taylor_midle}
    We have uniformly on $\theta>0$ and $\delta \le \frac{\theta}{3}$ 
    \begin{eqnarray*}
        f(\theta+\delta)-f(\theta)&=&\delta f'(\theta)+\frac{\delta^2}{2}f''(\theta)+O(\theta^{-2}\delta^3\exp(-\theta))\\
        j(\theta+\delta)-j(\theta)&=&\delta j'(\theta)+\bO{\theta^{-2}\delta^2\exp(-\theta)}.
    \end{eqnarray*}
\end{lem}

\begin{proof}
To prove lemma \ref{lem_taylor_midle}, we first use propositions \ref{prop_asym_zero} and \ref{prop_asym_infty}. For every $a<2$, we can deduce that, there is a constant $C$, such that for $k\in\{1,2\}$, we have 
\begin{equation*}
    |f^{(k+1)}(\theta)|\le C\theta^{-k}\exp(-a\theta)~~~\text{and}~~~|j^{(k)}(\theta)|\le C\theta^{-k}\exp(-a\theta).
\end{equation*}
Where the superscript stands for higher derivatives. Then we can bound the error term  in the Taylor expansion at the second order. We have $ \frac{2}{3}\theta\le \theta+\delta\le\frac{4}{3}\theta$, then 
\begin{equation*}
    \int_{\theta}^{\theta+\delta}f'''(x)(x-\theta)^2dx=\bO{\delta^3\theta^{-2}\exp\left(-\frac{2a}{3}\theta\right)}=\bO{\delta^3\theta^{-2}\exp(-\theta)}
\end{equation*}
we can treat $j_0$ similarly.
\end{proof}

We start with the following lemma
\begin{lem}
\label{lem_alpha_midle}
  Uniformly on $\theta \ge x^{-\frac{2}{3}}$, we have: 
    \begin{equation*}
        \alpha(n,g)=4\lambda(\theta)+\frac{4\lambda(\theta)}{n}\left (\theta-2+\frac{\theta^2}{2} f''(\theta)+\theta j'(\theta)-\frac{3}{2}\right)+\bO{\frac{1}{x^2}}.
    \end{equation*}
\end{lem}
\begin{proof}
Let $\delta=\frac{g}{n(n-1)}$, we can write:
\begin{equation*}
    \alpha(n,g)=\frac{2(2n-1)}{n+1}(1-\frac{1}{n})^{2g-2}\exp((n-1)(f(\theta+\delta)-f(\theta))-f(\theta)+j(\theta+\delta)-j(\theta)).
\end{equation*}
First, we have 
\begin{equation*}
    \frac{2(2n-1)}{n+1}=4-\frac{6}{n}+\frac{6}{n(n+1)}=4-\frac{6}{n}+\bO{\frac{(1+\theta)^2}{x^2}}
\end{equation*}
uniformly on $\theta$ and $x$. Using lemma \ref{lem_taylor_midle}, we have 
\begin{eqnarray*}
    \exp((n-1)(f(\theta+\delta)-f(\theta)))&=&\exp\left(\theta f'(\theta)+\frac{\theta^2}{2(n-1)}f''(\theta)+\bO{(n-1)\theta^{-2}\exp(-\theta)\delta^3}\right)\\
    &=&\exp\left(\theta f'(\theta)+\frac{\theta^2}{2(n-1)}f''(\theta)+\bO{\frac{\theta\exp(-\theta)}{(n-1)^2}}\right)\\
    &=&\exp(\theta f'(\theta))\left(1+\frac{\theta^2}{2n}f''(\theta)+\bO{\frac{\theta\exp(-\theta)}{n^2}}\right).\\
    &=&\exp(\theta f'(\theta))\left(1+\frac{\theta^2}{2n}f''(\theta)+\bO{\frac{(1+\theta)^3\exp(-\theta)}{x^2}}\right).
\end{eqnarray*}
Similarly, using lemma \ref{lem_taylor_midle}, uniformly on $\theta$, $j'(\theta)\theta =\bO{\exp(-\theta)}$. Then 
\begin{equation*}
    \exp(j(\theta+\delta)-j(\theta))=\exp\left(\frac{\theta}{n-1}j'(\theta)+\bO{\frac{\exp(-\theta)}{n^2}}\right)
    =1+\frac{\theta}{n}j'(\theta)+\bO{\frac{(1+\theta)^2\exp(-\theta)}{x^2}}.
\end{equation*}
If $n\ge 2$ and $g\ge 1$, we can write:
\begin{eqnarray*}
   \left(1-\frac{1}{n}\right)^{2g-2}&=&\exp\left((2g-2)\ln\left(1-\frac{1}{n}\right)\right)\\
   &=&\exp\left((2g-2)\left(-\frac{1}{n}-\frac{1}{2n^2} +\bO{\frac{1}{n^3}}\right)\right)\\
   &=&\exp\left(-2\theta+\frac{1}{n}(2-\theta)+\bO{\frac{\theta}{n^2}}\right).
\end{eqnarray*}
The error term is uniform, but in this range, $\frac{\theta}{n}$ is not bounded from above. First, assuming $\theta\le x^{\frac{1}{3}}$ we have $\frac{\theta}{n}=\bO{x^{-\frac{2}{3}}}$ which tends to $0$. Now, we can use Taylor expansion for $\exp$ and obtain
\begin{eqnarray*}
     \left(1-\frac{1}{n}\right)^{2g-2}
   &=&\exp(-2\theta)\left(1+\frac{1}{n}(2-\theta)+\bO{\frac{\theta^2}{n^2}}\right).
\end{eqnarray*}
Putting this together, assuming $\theta \le x^{\frac{1}{3}}$ and using $\lambda=\bO{\exp(-\theta)}$, we obtain:

\begin{eqnarray*}
    \alpha(n,g)&=&4\lambda \left(1-\frac{3}{2n}+\bO{\frac{1}{n^2}}\right)\left(1+\frac{1}{n}(2-\theta)+\bO{\frac{\theta^2}{n^2}}\right)\left( 1+\frac{1}{n}\left(\frac{\theta^2f''(\theta)}{2}+j'(\theta)\right)+\bO{\frac{(1+\theta)^3\exp(-\theta)}{x^2}}\right)\\
    &=&4\lambda\left(1+\frac{1}{n}\left(\frac{1}{2}-\theta\right)+\bO{\frac{(\theta+1)^4}{x^2}}\right)\left( 1+\frac{1}{n}\left(\frac{\theta^2f''(\theta)}{2}+j'(\theta)\right)+\bO{\frac{(1+\theta)^3\exp(-\theta)}{x^2}}\right)\\
&=&4\lambda\left(4\lambda+\frac{4\lambda}{n}\left(\frac{1}{2}-\theta\right)+\bO{\frac{(\theta+1)^4\exp(-\theta)}{x^2}}\right)\left( 1+\frac{1}{n}\left(\frac{\theta^2f''(\theta)}{2}+j'(\theta)\right)+\bO{\frac{(1+\theta)^3\exp(-\theta)}{x^2}}\right)\\
&=&4\lambda+\frac{4\lambda}{n}\left (\frac{1}{2}-\theta+\frac{\theta^2f''(\theta)}{2}+j'(\theta)\right)+\bO{\frac{(\theta+1)^4\exp(-\theta)}{x^2}}.
\end{eqnarray*}
We can bound the error term uniformly and obtain lemma \ref{lem_alpha_midle}. If $\theta\ge x^{\frac{1}{3}}$, the situation is simpler, we have the bound 
\begin{equation*}
    \exp(2\theta)\left (1-\frac{1}{n}\right)^{2g}\le 1.
\end{equation*}
We can factor $\lambda$ out of $\alpha$ and for $\theta\ge x^{\frac{1}{3}}$ the remaining term is bounded. Then, in this range, we have:
\begin{equation*}
    \alpha(n,g)=\bO{\lambda(\theta)}=\bO{\frac{1}{x^2}}.
\end{equation*}
Now, if we take the RHS of lemma \ref{lem_alpha_midle}, we see that the main term grows like $\bO{\theta \exp(-\theta)}$, then, the RHS is indeed a $\bO{\frac{1}{x^2}}$. We then obtain lemma \ref{lem_alpha_midle} in the case $\theta\le x^{\frac{1}{3}}$, all the terms in the formula are way smaller than that $\bO{\frac{1}{x^2}}$.
\end{proof}
\begin{lem}
\label{lem_beta_midle}
    Uniformly on $\theta \ge x^{-\frac{2}{3}}$, we have: 
    \begin{equation*}
        \beta(n,g)=\exp(-f'(\theta))+\frac{\exp(-f'(\theta))}{n}\left ( \frac{f''(\theta)}{2}-j'(\theta)\right)+\bO{\frac{1}{x^{\frac{4}{3}}}}.
    \end{equation*}
\end{lem}
\begin{proof}
    This time, let $\delta=-\frac{1}{n}$, we want to approximate
    \begin{equation*}
        \beta(n,g)=\frac{C_{g-1}}{C_g}\exp\left( n(f(\theta+\delta)-f(\theta))+j(\theta+\delta)-j(\theta)\right).
    \end{equation*}
    First, by Stirling formula we have 
    \begin{equation*}
        \frac{C_{g-1}}{C_g}=1+\bO{\frac{1}{g^2}}=1+\bO{\frac{(\theta+1)^2}{\theta^2x^2}}.
    \end{equation*}
    Using lemma \ref{lem_taylor_midle}, we obtain
\begin{eqnarray*}
    \exp(n(f(\theta+\delta)-f(\theta)))&=&\exp\left(-f'+\frac{f''}{2n}+\bO{\frac{\exp(-\theta)}{\theta^2n^2}}\right)\\
    &=&\exp(-f')\left(1+\frac{f''}{2n}+\bO{\frac{\exp(-\theta)}{\theta^2 n^2}}\right)
\end{eqnarray*}
because $\frac{f''}{2n}=\bO{\frac{\exp(-\theta)}{\theta n}}$. Similarly, we have
\begin{eqnarray*}
     \exp(j(\theta+\delta)-j(\theta))&=&\exp\left(-\frac{j'}{n}+\bO{\frac{\exp(-\theta)}{\theta^2n^2}}\right)\\
    &=&\left(1-\frac{j'}{n}+\bO{\frac{\exp(-\theta)}{\theta^2 n^2}}\right).
\end{eqnarray*}
Then we obtain:
    \begin{eqnarray*}
       \beta(n,g)&=&\exp(-f')\left(1+\bO{\frac{(\theta+1)^2}{\theta^2x^2}}\right)\left(1+\frac{f''}{2n}+\bO{\frac{\exp(-\theta)}{\theta^2 n^2}}\right)\left(1-\frac{j'}{n}+\bO{\frac{\exp(-\theta)}{\theta^2 n^2}}\right)\\
       &=&\exp(-f')\left(1+\frac{f''}{2n}-\frac{j'}{n}+\bO{\frac{(\theta+1)^2}{\theta^2x^2}}\right).
\end{eqnarray*}
Using proposition \ref{prop_asym_zero}, we have $\frac{\exp(-f')}{\theta}=\bO{1}$ and using proposition \ref{prop_asym_infty}, $\exp(-f')$ is bounded near $+\infty$. Then $\exp(-f')\frac{\theta+1}{\theta}$ is bounded on $\Rp$ and we can write
  \begin{eqnarray*}
       \beta(n,g)&=&\exp(-f')\left(1+\frac{f''}{2n}-\frac{j'}{n}\right)+\bO{\frac{(\theta+1)}{\theta x^2}}.
\end{eqnarray*}
In the worst case, we have $\frac{\theta+1}{\theta x}=\bO{\frac{1}{x^{\frac{1}{3}}}}$,
and then 
\begin{equation*}
      \beta(n,g)=\exp(-f')+\frac{\exp(-f')}{2}\left(\frac{f''}{2}-j'\right)+\bO{\frac{1}{x^{\frac{4}{3}}}}.
\end{equation*}
   
\end{proof}

\subsubsection{Proof of proposition \ref{prop_alpha_beta}}
We synthesize the results of the last subsections. According to lemma \ref{lem_low}, if $\theta\le x^{-\frac{2}{3}}$ we obtain
\begin{equation*}
    \alpha+\beta=1=\bO{\frac{1}{x^{\frac{4}{3}}}}.
\end{equation*}
In the second case, in the light of lemmas \ref{lem_alpha_midle} and \ref{lem_beta_midle}, we can see that 
\begin{equation*}
    \alpha+\beta=4\lambda+\exp(-f')+\frac{1}{n}\left(4\lambda\left(\frac{1}{2}-\theta+\frac{\theta^2}{2}f''+\theta j'\right)+\exp(-f')\left(\frac{f''}{2}-j'\right)\right)+\bO{\frac{1}{x^{\frac{4}{3}}}}.
\end{equation*}
We can conclude by using equations \ref{eq_heuristic_f} and \ref{eq_def_j}, since we have 
\begin{eqnarray*}
   1&=& 4\lambda+\exp(-f')\\
   0&=& 4\lambda\left(\frac{1}{2}-\theta+\frac{\theta^2}{2}f''+\theta j'\right)+\exp(-f')\left(\frac{f''}{2}-j'\right).
\end{eqnarray*}
\subsection{Properties of $s$ and boundary conditions}

It remains to verify the other assumptions. We start by establishing assumption~\ref{assum_good_starting_point}.

\begin{proposition}\label{prop_s}
As $\nto$, uniformly in $g$, we have
    \begin{equation*}
        s(n,g)=\bO{\frac{1}{n^2}}.
    \end{equation*}

\end{proposition}

\begin{proof}
    Note that $s(g,n)=\frac{\beta(n,g)}{n^2}$. By proposition~\ref{prop_alpha_beta}, we know that $\beta(n,g)=O(1)$ as \nto, uniformly in $g$, this entails the result.
\end{proof}

Now, from~\eqref{eq_def_Omega} and proposition~\ref{prop_asym_infty}, we get that for fixed $n$ and $g\to\infty$, we have
\begin{equation}
    \Omega(n,g)\sim C(n)n^{2g}
\end{equation}
where $C(n)$ is a positive constant depending on $n$. 
Noting that the numbers $\Omega(n,g)$ are always strictly positive, we can immediately deduce that assumption~\ref{assum_s} is satisfied.

Since for all $g$, $E(1,g)=1$, this also implies that $Q(1,g)$ is bounded from above. Finally, one can check from~\eqref{eq_recursion} that $E(n,0)=\frac{(2n)!}{(n+1)!n!}\sim \frac{4^n}{n^{\frac{3}{2}}\sqrt{2\pi}}$, and from~\eqref{eq_def_Omega_zero} we obtain that $Q(n,0)\to 1$ as $\nto$, therefore assumption~\ref{assum_boundary_conditions} is satisfied. 
\printbibliography

\end{document}